\documentclass[11pt,a4paper]{article}
\usepackage{a4wide}

\usepackage[T1]{fontenc}
\usepackage{lmodern}
\usepackage{amsthm,amsmath,amssymb}
\usepackage{hyperref}
\usepackage{microtype}
\usepackage{hyperref}
\usepackage{graphicx}
\usepackage{enumitem}

\newtheorem{thm}{Theorem}%

\newtheorem{lemma}[thm]{Lemma}
\newtheorem{cor}[thm]{Corollary}
\newtheorem{proposition}[thm]{Proposition}

\newtheorem{problem}[thm]{Problem}
\newtheorem{conjecture}[thm]{Conjecture}

\theoremstyle{remark}

\if10    
\usepackage[mathlines]{lineno}
\newcommand*\patchAmsMathEnvironmentForLineno[1]{%
  \expandafter\let\csname old#1\expandafter\endcsname\csname #1\endcsname
  \expandafter\let\csname oldend#1\expandafter\endcsname\csname end#1\endcsname
  \renewenvironment{#1}%
     {\linenomath\csname old#1\endcsname}%
     {\csname oldend#1\endcsname\endlinenomath}}%
\newcommand*\patchBothAmsMathEnvironmentsForLineno[1]{%
  \patchAmsMathEnvironmentForLineno{#1}%
  \patchAmsMathEnvironmentForLineno{#1*}}%
\AtBeginDocument{%
\patchBothAmsMathEnvironmentsForLineno{equation}%
\patchBothAmsMathEnvironmentsForLineno{align}%
\patchBothAmsMathEnvironmentsForLineno{flalign}%
\patchBothAmsMathEnvironmentsForLineno{alignat}%
\patchBothAmsMathEnvironmentsForLineno{gather}%
\patchBothAmsMathEnvironmentsForLineno{multline}%
}
\linenumbers
\fi

\def\inst#1{$^{#1}$}

\begin{document}

\title{On off-diagonal ordered Ramsey numbers of nested matchings}

\author{Martin Balko\inst{1} \thanks{Martin Balko was supported by the grant no.~19-04113Y of the Czech Science Foundation (GA\v{C}R) and by the Center for Foundations of Modern Computer Science (Charles University project UNCE/SCI/004). This article is part of a project that has received funding from the European Research Council (ERC) under the European Union's Horizon 2020 research and innovation programme (grant agreement No. 810115).} 
\and
Marian Poljak \inst{1} \thanks{Marian Poljak was supported by the grant
SVV–2020-260 578.} 
}

\maketitle

\begin{center}
{\footnotesize
\inst{1} 
Department of Applied Mathematics, \\
Faculty of Mathematics and Physics, Charles University, Czech Republic \\
\texttt{balko@kam.mff.cuni.cz}, \texttt{marian@kam.mff.cuni.cz}
}
\end{center}

\begin{abstract}
For two graphs $G^<$ and $H^<$ with linearly ordered vertex sets, the \emph{ordered Ramsey number} $r_<(G^<,H^<)$ is the minimum $N$ such that every red-blue coloring of the edges of the ordered complete graph on $N$ vertices contains a red copy of $G^<$ or a blue copy of $H^<$.

For a positive integer $n$, a \emph{nested matching} $NM^<_n$ is the ordered graph on $2n$ vertices with edges $\{i,2n-i+1\}$ for every $i=1,\dots,n$.
We improve bounds on the ordered Ramsey numbers $r_<(NM^<_n,K^<_3)$ obtained by Rohatgi, we disprove his conjecture by showing $4n+1 \leq r_<(NM^<_n,K^<_3) \leq (3+\sqrt{5})n+1$ for every $n \geq 6$, and we determine the numbers $r_<(NM^<_n,K^<_3)$ exactly for $n=4,5$.
As a corollary, this gives stronger lower bounds on the maximum chromatic number of $k$-queue graphs for every $k \geq 3$.
We also prove $r_<(NM^<_m,K^<_n)=\Theta(mn)$ for arbitrary $m$ and $n$.

We expand the classical notion of Ramsey goodness to the ordered case and we attempt to characterize all connected ordered graphs that are $n$-good for every $n\in\mathbb{N}$.
In particular, we discover a new class of ordered trees that are $n$-good for every $n \in \mathbb{N}$, extending all the previously known examples.
\end{abstract}

\section{Introduction}
Ramsey theory is devoted to the study of the minimum size of a system that guarantees the existence of a highly organized subsystem. 
Given graphs $G$ and $H$, their \emph{Ramsey number} $r(G,H)$ is the smallest $N \in \mathbb{N}$ such that any two-coloring of the edges of $K_N$ contains either $G$ as a red subgraph or $H$ as a blue subgraph of $K_N$.
The case $G=H$ is called the \emph{diagonal} case and in this case we use the abbreviation $r(G) = r(G,G)$. 

The growth rate of Ramsey numbers has been of interest to many researchers.
In general, it is notoriously difficult to find tight estimates on Ramsey numbers.
For example, despite many efforts, the best known bounds on $r(K_n)$ are essentially
\begin{equation}
\label{eq-ramseyBounds}
2^{n/2} \leq r(K_n)\leq 2^{2n}
\end{equation}
obtained by Erd\H{o}s and Szekeres~\cite{erdosSzekeres35}, although some smaller term improvements are known.
For a more comprehensive survey we can refer the reader to~\cite{recent_developments}. 

In this paper, we study ordered graphs.
An \emph{ordered graph} $G^<$ on $n$ vertices is a graph whose vertex set is $[n]:=\{1,\dots, n\}$ and it is ordered by the standard ordering $<$ of integers.
For an ordered graph $G^<$, we use $G$ to denote its unordered counterpart.
An ordered graph $G^<$ on $[n]$ is an \emph{ordered subgraph} of another ordered graph $H^<$ on $[N]$ if there exists a mapping $\phi: [n] \rightarrow [N]$ such that $\phi(i) < \phi(j)$ for $1 \leq i < j \leq n$ and also $\{\phi(i),\phi(j)\}$ is an edge of $H^<$ whenever $\{i,j\}$ is an edge of $G^<$.
Definitions that are often stated for unordered graphs, such as vertex degrees, degeneracy, colorings, and so on, have their natural analogues for ordered graphs. Note that, for every $n\in\mathbb{N}$, there is a unique complete ordered graph $K^<_n$.

Motivated by connections to classical results such as the Erd\H{o}s--Szekeres theorem on monotone subsequences~\cite{erdosSzekeres35},
various researchers \cite{balko,conlon} recently initiated the study of Ramsey numbers of ordered graphs.
Given two ordered graphs $G^<$ and $H^<$, the \emph{ordered Ramsey number} $r_<(G^<,H^<)$ is defined as the smallest $N$ such that any two-coloring of the edges of~$K^<_N$ contains either $G^<$ as a red ordered subgraph or $H^<$ as a blue ordered subgraph.

Observe that for any two ordered graphs $G^<_1$ and $G^<_2$ on $n_1$ and $n_2$ vertices, respectively, we have
$r(G_1, G_2) \leq r_<(G^<_1, G^<_2) \leq r(K_{n_1},K_{n_2})$.
Thus, by~\eqref{eq-ramseyBounds}, the number $r_<(G^<_1, G^<_2)$ is finite and, in particular, $r_<(G^<)$ is at most exponential in the number of vertices for every ordered graph $G^<$.

It is known that for dense graphs, there is not a huge difference in the growth rate of their ordered and unordered Ramsey numbers~\cite{balko,conlon}.
On the other hand, ordered Ramsey numbers of sparse ordered graphs behave very differently from their unordered counterparts.
For example,  Ramsey numbers of \emph{matchings} (that is, graphs with maximum degree $1$) are clearly linear in the number of their vertices. However, it was proved independently in~\cite{balko, conlon} that there exist ordered matchings whose diagonal ordered Ramsey numbers grow superpolynomially. 

\begin{thm}[\cite{balko,conlon}]
\label{superpolynomial_matchings}
There are arbitrarily large ordered matchings $M^<$ on $n$ vertices that satisfy
\[r_<(M^<) = n^{\Omega\left(\frac{\log n}{\log{\log n}}\right)}.\]
\end{thm}

The bound from Theorem~\ref{superpolynomial_matchings} is quite close to the truth, as Conlon, Fox, Lee and Sudakov~\cite{conlon} proved that $r_<(G^<,K^<_n) = 2^{O(d\log^2{(2n/d)})}$ for every ordered graph $G^<$ on $n$ vertices with degeneracy $d$.
In particular, we have $r_<(G^<) \leq r_<(G^<,K^<_n) = n^{O(\log{n})}$ if $G^<$ has its maximum degree bounded by a constant.

There has also been a keen interest in studying the off-diagonal ordered Ramsey numbers.
Conlon, Fox, Lee and Sudakov~\cite{conlon} proved that there exist ordered matchings $M^<$ on $n$ vertices such that $r_<(M^<,K^<_3) = \Omega\left(\left(\frac{n}{\log n}\right) ^{\frac{4}{3}}\right)$. On the other hand, the best known upper bound on $r_<(M^<,K^<_3)$ is
\[r_<(M^<,K^<_3) \leq r_<(K^<_n,K^<_3)=r(K_n,K_3) = O\left(  \frac{n^2}{\log n} \right),\]
which follows from the well-known result $r(K_n,K_3) = O\left(\frac{n^2}{\log n} \right)$~\cite{odhadtrojuhelnik}, which is tight~\cite{odhadtrojuhelnik2}. 
Note that the first inequality only uses the fact that $M^<$ is an ordered subgraph of $K^<_n$ and does not utilize any special properties of ordered matchings such as its sparseness.
Conlon, Fox, Lee and Sudakov~\cite{conlon} expect that
the upper bound is far from being optimal and posed the following problem.

\begin{problem}[\cite{conlon}]
\label{subquad}
Does there exist an $\varepsilon>0$ such that any ordered matching $M^<$ on $n \in \mathbb{N}$ vertices satisfies
$r_<(M^<,K^<_3) = O\left( n^{2-\varepsilon}\right)$?
\end{problem}

Problem~\ref{subquad} remains open, but there was some progress obtained by Rohatgi~\cite{rohatgi}, who resolved some special cases of this problem.
In particular, he proved that if the edges of an ordered matching $M^<$ do not cross, then the ordered Ramsey number $r_<(M^<,K^<_3)$ is almost linear.
The basic building block of the proof of this result is formed by so-called nested matchings.
For $n\in\mathbb{N}$, a \emph{nested matching} (or a \emph{rainbow}) $NM^<_n$ is the ordered matching on $2n$ vertices with edges $\{i,2n-i+1\}$ for every $i \in [n]$.

Rohatgi~\cite{rohatgi} determined the off-diagonal ordered Ramsey numbers of nested matchings up to a constant factor.

\begin{proposition}[\cite{rohatgi}]
\label{nm}
For every $n\in\mathbb{N}$, we have
$$4n-1 \leq r_<(NM^<_n,K^<_3) \leq 6n.$$
\end{proposition}

He believed that the upper bound is far from optimal and posed the following conjecture, which he verified for $n \in \{1,2,3\}$.

\begin{conjecture}[\cite{rohatgi}]
\label{nmconj}
For every $n\in\mathbb{N}$, we have
\[r_<(NM^<_n,K^<_3) = 4n-1.\]
\end{conjecture} 

The ordered graphs that do not contain $NM^<_m$ as an ordered subgraph for some $m \in \mathbb{N}$ are known to be equivalent to so-called \emph{$(m-1)$-queue graphs}~\cite{arched} and, in particular, 1-queue graphs correspond to \emph{arched-leveled-planar graphs}~\cite{arched}.
As we will see, estimating the ordered Ramsey numbers $r_<(NM^<_m,K^<_n)$ is connected to extremal questions about $(m-1)$-queue graphs.
In particular, there is a close connection to the problem of Dujmović and Wood~\cite{dujWoo10} about determining the chromatic number of such graphs.

\begin{problem}[\cite{dujWoo10}]
\label{dujWoodProblem}
What is the maximum chromatic number $\chi_k$ of a $k$-queue graph?
\end{problem} 

Dujmović and Wood~\cite{dujWoo10} note that $\chi_k \in \{2k+1,\dots,4k\}$ and they prove that the lower bound is attainable for $k=1$.

\section{Our results}

In this paper, we also focus on off-diagonal ordered Ramsey numbers.
In particular, we improve and generalize the bounds on $r_<(NM^<_n,K^<_3)$ and we disprove Conjecture~\ref{nmconj}.
We also consider ordered Ramsey numbers $r_<(G^<,K_n)$ for general connected ordered graphs $G^<$ and we introduce the concept of Ramsey goodness of ordered graphs. 
Finally, we pose several new open problems.

\subsection{Nested matchings versus complete graphs}
\label{subsec-nested}

First, we improve the leading constant in the upper bound from Proposition~\ref{nm} and thus show that the bound by Rohatgi is indeed not tight. However, we believe that our estimate can be improved as well.

\begin{thm}
\label{prop:nm5.3}
For every $n\in\mathbb{N}$, we have
\[r_<(NM^<_n,K^<_3)\leq \left(3+\sqrt 5\right)n +1< 5.3n+1.\]
\end{thm}

Next, we disprove Conjecture~\ref{nmconj} by showing $r_<(NM^<_n,K^<_3) > 4n-1$ for every $n \geq 4$.
For $n \in \{4,5\}$, we determine $r_<(NM^<_n,K^<_3)$ exactly.

\begin{thm}
\label{thm:nm_counterexamples}
For every $n\geq6$, we have \[r_<(NM^<_n,K^<_3) \geq 4n+1.\]
Moreover, $r_<(NM^<_4,K^<_3) = 16$ and $r_<(NM^<_5,K^<_3) = 20$. 
\end{thm}

We prove the lower bound $r_<(NM^<_n,K^<_3) \geq 4n+1$ by constructing a specific red-blue coloring of the edges of $K^<_{4n}$ that avoids a red copy of $NM^<_n$ and a blue copy of $K^<_3$. 
To determine $r_<(NM^<_4,K^<_3)$ and $r_<(NM^<_5,K^<_3)$ exactly, we use a computer-assisted proof based on SAT solvers. For more details about the use of SAT solvers for finding avoiding colorings computationally, we refer the reader to the bachelor's thesis of the second author~\cite{bakalarka}. The utility we developed for computing ordered Ramsey numbers $r_<(G^<,H^<)$ for small ordered graphs $G^<$ and $H^<$ is publicly available~\cite{utility}.

By performing the exhaustive computer search, we know that there are only 326 red-blue colorings of the edges of $K^<_{15}$ without a red copy of $NM^<_4$ and a blue copy of $K^<_3$.
They all share the same structure except for 6 red edges that can be switched to blue while not introducing a blue triangle. 
Using the same computer search, we were able to find many red-blue colorings of the edges of $K^<_{19}$ without a red copy of $NM^<_5$ and a blue copy of $K^<_3$, some of which even had certain symmetry properties.
There were no such symmetric colorings on 15 vertices, which suggests that the lower bound on $r_<(NM^<_n,K^<_3)$ might be further improved for larger values of $n$.

Using the lower bounds from Theorem~\ref{thm:nm_counterexamples}, we can address Problem~\ref{dujWoodProblem} about the maximum chromatic number $\chi_k$ of $k$-queue graphs.
In particular, we can improve the lower bound $\chi_k \geq 2k+1$ by $1$ for any $k \geq 3$; see Subsection~\ref{subsec-nm2} for a proof.

\begin{cor}
\label{cor-chromatic}
For every $k \geq 3$, the maximum chromatic number of $k$-queue graphs is at least $2k+2$.
\end{cor}

We recall that the maximum chromatic number $\chi_1$ of $1$-queue graphs is $3$~\cite{dujWoo10}.
We use this result to prove the exact formula for the off-diagonal ordered Ramsey numbers $r_<(NM^<_2,K^<_n)$ of nested matchings with two edges.

\begin{thm}
\label{prop:nm2}
For every $n\in\mathbb{N}$, we have $r_<(NM^<_2,K^<_n)=3n-2$.
\end{thm}

For general nested matchings versus arbitrarily large complete graphs, we can determine the asymptotic growth rate of their ordered Ramsey numbers, generalizing the linear bounds from Proposition~\ref{nm} and Theorem~\ref{prop:nm5.3}.

\begin{thm}
\label{thm:linearni_matchingy}
For every $m,n\in\mathbb{N}$, we have
\[r_<(NM^<_m,K^<_{n+1}) = \Theta(mn).\]
\end{thm}

\subsection{Ramsey goodness for ordered graphs}
\label{subsec:goodness}

To obtain the lower bound $r(G,K_n) \geq (m-1)(n-1)+1$ for a connected graph $G$ on $m$ vertices, one might consider a simple construction that is usually attributed to Chv\'{a}tal and Harary~\cite{chvaHa72}.
Take $n-1$ red cliques, each with $m-1$ vertices, and connect vertices in different red cliques by blue edges. 
For some graphs $G$, this lower bound is the best possible and such graphs are called \emph{(Ramsey) $n$-good}.
That is, a connected graph $G$ on $m$ vertices is $n$-good if $r(G,K_n)=(m-1)(n-1)+1$. We call a graph \emph{good} if it is $n$-good for all $n\in\mathbb{N}$.
A famous result by Chvátal~\cite{good_stromy} states that all trees are good.

Studying $n$-good graphs is a well-established area in extremal combinatorics. 
Despite this, to the best of our knowledge, Ramsey goodness has not been considered for ordered graphs. 
Motivated by our results from Subsection~\ref{subsec-nested}, we thus extend the definition of good graphs to ordered graphs and we attempt to characterize all good connected ordered graphs.
A connected ordered graph $G^<$ on $m$ vertices is \emph{$n$-good} if $r_<(G^<,K^<_n)=(m-1)(n-1)+1$. A connected ordered graph is \emph{good} if it is $n$-good for all $n\in\mathbb{N}$.

A generalization of the well-known Erd\H{o}s--Szekeres theorem on monotone subsequences states that $r_<(P^<_m,K^<_n)=(m-1)(n-1)+1$ for every $n\in\mathbb{N}$ and every \emph{monotone path $P^<_m$}~\cite{choPo02}, which is an ordered path on $m$ vertices where edges connect consecutive vertices in $<$.
In other words, any monotone path is good, which gives a first example of good ordered graphs. 
Note, however, that not all ordered paths are good, which follows immediately from Theorem~\ref{superpolynomial_matchings}.

First, we prove some basic properties of good ordered graphs, some of them resembling their unordered counterparts. 
It can be shown similarly as in the unordered case that if a connected ordered graph $G^<$ is $(n+1)$-good, then it is $n$-good.

Let $G^<$ be an ordered graph containing an ordered cycle $C^<$ (a cycle equipped with a linear vertex ordering) as an ordered subgraph.
It is known that, for every cycle $C_l$ on $l \geq 3$ vertices and for $n$ going to infinity, the Ramsey number $r(C,K_n)$ grows superlinearly with $n$~\cite{neusp_cykly_odhad,odhadtrojuhelnik2}.
Since $r_<(G^<,K^<_n)\geq r_<(C^<,K^<_n) \geq r(C,K_n)$, the number $r_<(G^<,K^<_n)$ is also superlinear in $n$ and thus the ordered graph $G^<$ cannot be good.
We thus obtain the following result that limits good ordered graphs to ordered trees.

\begin{proposition}
\label{thm:with_cycles_not_good}
Every good ordered graph is an ordered tree.
\end{proposition}

In our attempt to determine which ordered trees are good, we discovered a class of good ordered trees, which significantly extends the example with monotone paths.
In order to describe this new class, we need to introduce some notation.

An \emph{ordered star graph} $S^<_{l,r}$ is an ordered graph on $r+l-1$ vertices such that the $l$th vertex in the vertex ordering is adjacent to all other vertices and there are no other edges; see part~(a) of Figure~\ref{fig:caterpillar}.
We call an ordered star $S^<_{l,r}$ \emph{one-sided}, if $l=1$ or $r=1$.

For any two ordered graphs $G^<$ and $H^<$ on $m$ and $n$ vertices, respectively, the \emph{join} $G^< + H^<$ is an ordered graph on $m+n -1$ vertices constructed by identifying the leftmost vertex of $H^<$ with the rightmost vertex of $G^<$; see part~(b) of Figure~\ref{fig:caterpillar}.
The join operation is associative and if $G^<$ and $H^<$ are both connected, then $G^< + H^<$ is connected as well.

\begin{figure}[ht]
\centering
\includegraphics{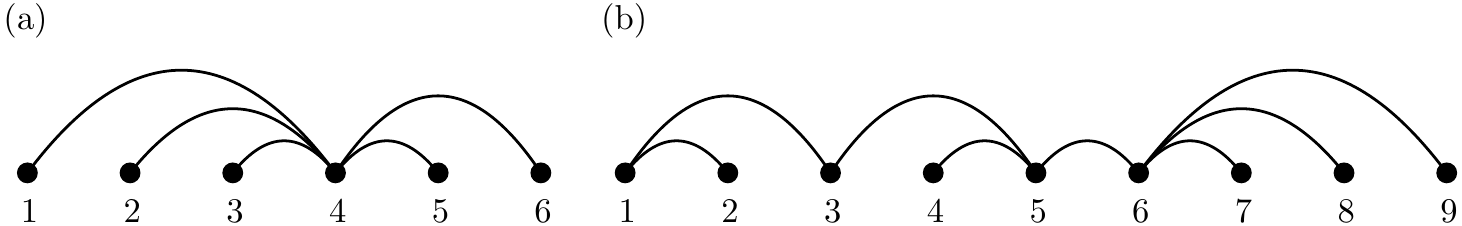}
\caption{(a) The ordered star graph $S_{4,3}^<$. (b) An example of a monotone caterpillar graph $S^<_{1,3}+S^<_{3,1}+S^<_{1,2}+S^<_{1,4}$. Note that this ordered graph can be also written as, for example, $S^<_{1,3} + S^<_{3,2}+S^<_{1,4}$.}
\label{fig:caterpillar}
\end{figure}

The following result gives a construction of good ordered graphs based on the join operation.

\begin{thm}
\label{thm:expansion}
For all $n, r,l\in\mathbb{N}$, if a connected ordered graph $G^<$ is $n$-good, then the ordered graphs $G^< + S^<_{1,r}$, $G^< + S^<_{l,1}$, $S^<_{l,1}+G^<$, and $S^<_{1,r}+G^<$ are also $n$-good.
\end{thm}

Theorem~\ref{thm:expansion} immediately implies that every ordered star graph is good. 
More generally, it follows that all ordered trees from the following class are good.
An ordered graph $G^<$ is a \emph{monotone caterpillar graph} if there exist positive integers $n,l_1,\dots,l_n, r_1,\dots,r_n$ such that $l_i=1$ or $r_i=1$ for each $i \in [n]$ and
$G^< = S^<_{l_1,r_1} + \dots +S^<_{l_n,r_n}$.
In other words, if $G^<$ can be obtained by performing joins on one-sided ordered star graphs; see part~(b) of Figure~\ref{fig:caterpillar}.
Note that monotone paths and ordered stars are all monotone caterpillar graphs.

\begin{cor}
All monotone caterpillar graphs are good.
\end{cor}

Computer experiments based on our SAT solver based utility~\cite{utility} proved that all good ordered graphs up to 6 vertices are monotone caterpillar graphs.
We believe that there are no other good ordered graphs; see Conjecture~\ref{conj:caterpillar}. 

To get a better understanding of good ordered graphs, we prove an alternative characterization of monotone caterpillar graphs stated in terms of forbidden ordered subgraphs.

\begin{proposition}
\label{prop:forbidden_caterpillar}
A connected ordered graph $G^<$ is a monotone caterpillar graph if and only if $G^<$ does not contain any of the four ordered graphs from Figure~\ref{fig:forbidden_caterpillar} as an ordered subgraph.
\end{proposition}

\begin{figure}[ht]
\centering
\includegraphics{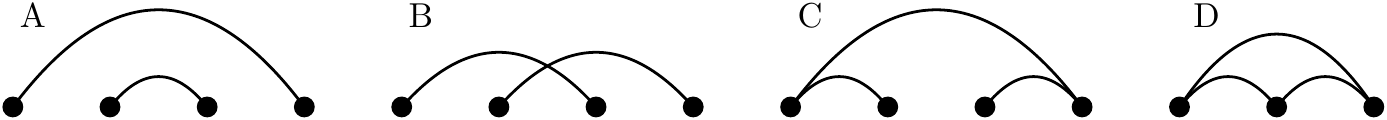}
\caption{Any ordered graph that does not contain any of these four ordered graphs as an ordered subgraph is a monotone caterpillar graph.}
\label{fig:forbidden_caterpillar}
\end{figure}

We note that if we assume that $G^<$ is an ordered tree, then we can leave out $D$ in Figure~\ref{fig:forbidden_caterpillar} and the characterization still holds.
It follows from Proposition~\ref{prop:forbidden_caterpillar} that monotone caterpillar graphs can be also characterized using nine forbidden \emph{connected} ordered subgraphs.
It suffices to extend the two disconnected ordered graphs $A$ and $B$ from Figure~\ref{fig:forbidden_caterpillar} into connected ordered graphs using a simple case analysis.

\begin{cor}
\label{cor:caterpillar_char}
An ordered tree $G^<$ is a monotone caterpillar graph if and only if it does not contain any ordered graph from Figure~\ref{fig:forbidden_caterpillar2} as an ordered subgraph.
\end{cor}

\begin{figure}[ht]
\centering
\includegraphics{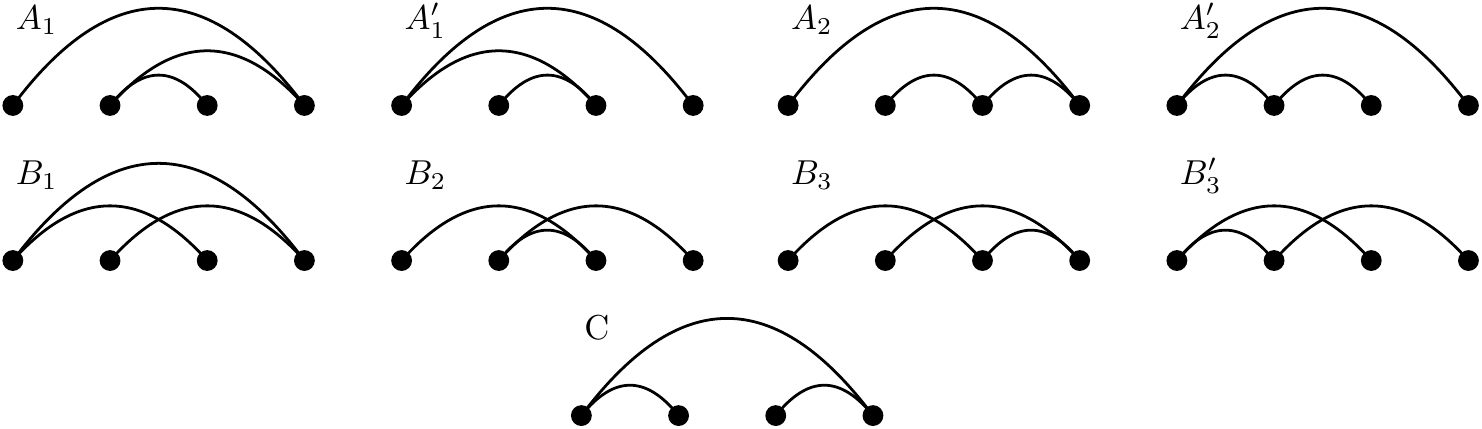}
\caption{Forbidden connected ordered subgraphs for monotone caterpillar graphs. Also, these are all ordered trees on four vertices that are not monotone caterpillar graphs.}
\label{fig:forbidden_caterpillar2}
\end{figure}

It follows from our computer-assisted proofs that none of the nine ordered graphs from Figure~\ref{fig:forbidden_caterpillar2} is good.

\subsection{Open problems}

By Theorem~\ref{thm:nm_counterexamples}, the value of $r_<(NM^<_m,K^<_n)$ does not match the value of the ordered Ramsey number of an $n$-good ordered graph on $2m$ vertices if $m \geq 4$ already in the case $n=3$.
However, Theorem~\ref{prop:nm2} implies that the two values are the same for $m=2$ and every $n \geq 3$.
In other words, $NM^<_2$ can be considered good, but the ordered graphs $NM^<_m$ with $m \geq 4$ cannot (although neither of these ordered graphs is connected and the definition of good ordered graphs is stated only for connected ordered graphs).
We do not know where the truth lies for $NM^<_3$, although we have computationally verified that $r_<(NM^<_3,K^<_n) = 5n-4$ for $n \leq 6$. 
These results are closely connected to proper colorings of ordered graphs that do not contain nested matchings as ordered subgraphs; see Subsection~\ref{subsec-nm2}, for an example.
In particular, a related problem is to decide whether the maximum chromatic number $\chi_2$ of $2$-queue graphs is at most $5$.

\begin{problem}
\label{nm3}
Is it true that $r_<(NM^<_3,K^<_n) = 5n-4$ for every positive integer $n$?
\end{problem}

Despite our efforts, we still lack complete characterization of all ordered Ramsey good graphs.
We exhaustively searched for good graphs with small number of vertices and we verified that there are no $4$-good graphs with at most 6 vertices other than monotone caterpillar graphs.
Based on our experimental results, we believe that all good ordered graphs are monotone caterpillar graphs and we pose the following conjecture.

\begin{conjecture}
\label{conj:caterpillar}
A connected ordered graph $G^<$ is good if and only if it is a monotone caterpillar graph.
\end{conjecture}

We do not know whether each connected ordered subgraph of an $n$-good ordered graph is $n$-good.
If this was true, then Corollary~\ref{cor:caterpillar_char} together with our computer-assisted results would imply that Conjecture~\ref{conj:caterpillar} is true as well.

\section{Nested matchings versus complete graphs}

This section contains proofs of all statements from Subsection~\ref{subsec-nested}, that is, all results about off-diagonal ordered Ramsey numbers of nested matchings versus complete graphs.
In particular, we prove Theorems~\ref{prop:nm5.3}, \ref{thm:nm_counterexamples}, \ref{thm:linearni_matchingy}, and \ref{prop:nm2}.

\subsection{Proof of Theorem~\ref{prop:nm5.3}}
\label{subsec-prop:nm5.3}

To prove Theorem~\ref{prop:nm5.3}, that is, to show that $r_<(NM^<_n,K^<_3)\leq \left(3+\sqrt 5\right)n +1< 5.3n+1$ for every $n \in \mathbb{N}$, we first state a lemma about the number of edges in an ordered graph that does not contain $NM^<_n$ as an ordered subgraph.

We will assume that an ordered graph $G^<$ on $n$ vertices is represented by a matrix. 
The \emph{matrix representation} $A$ of $G^<$ is an $n \times n$ $\{0,1\}$-matrix such that the entry of $A$ on position $(i,j)$ is 1 if and only if $i<j$ and $\{i,j\}$ is an edge of $G^<$.
Sometimes we will not distinguish between $A$ and $G^<$ and we identify the edges of $G^<$ with positions of 1-entries in $A$.
Also, given a red-blue coloring $\chi$ of the edges of $K^<_n$, we define the \emph{matrix representation of $\chi$} to be the matrix representation of the ordered graph formed by red edges in $\chi$.
In particular, red edges correspond to 1-entries and blue edges correspond to 0-entries in such matrix.

We state the following lemma about the maximum number of edges in an ordered graph without a copy of $NM^<_n$.
This result was proved by Dujmović and Wood~\cite{dujWoo10} in their study of queue graphs. Nevertheless, we include its proof for completeness and present it using the matrix representations, which are more suitable for representing the red-blue colorings of complete ordered graphs.

\begin{lemma}[\cite{dujWoo10}]
\label{lemma:disjoint_paths}
For every $n\in\mathbb{N}$, if an ordered graph $G^<$ on $N$ vertices with $N \geq 2n$ does not contain $NM^<_n$ as an ordered subgraph, then the number of edges in $G^<$ is at most $(n-1)(2N-2n+1)$. 
Moreover, this upper bound is tight.
\end{lemma}
\begin{proof}
Let $G^<$ be an ordered graph with the vertex set $[N]$ such that $G^<$ does not contain $NM^<_n$ as an ordered subgraph.
Let $A$ be a matrix representation of $G^<$.
For $k \in \{3,\dots,2N-1\}$, we define the $k$th \emph{anti-diagonal} of $A$ to be the set of positions $(i,j)$ of $A$ such that $i+j=k$, and $i<j$; see part~(a) of Figure~\ref{fig:anti-diagonals}. 
Note that there are exactly $2N-3$ anti-diagonals and each one of them contains at most $(n-1)$ 1-entries, as $n$ 1-entries in an anti-diagonal give a copy of the nested matching $NM^<_n$ in $G^<$.

\begin{figure}[ht]
\centering
\includegraphics{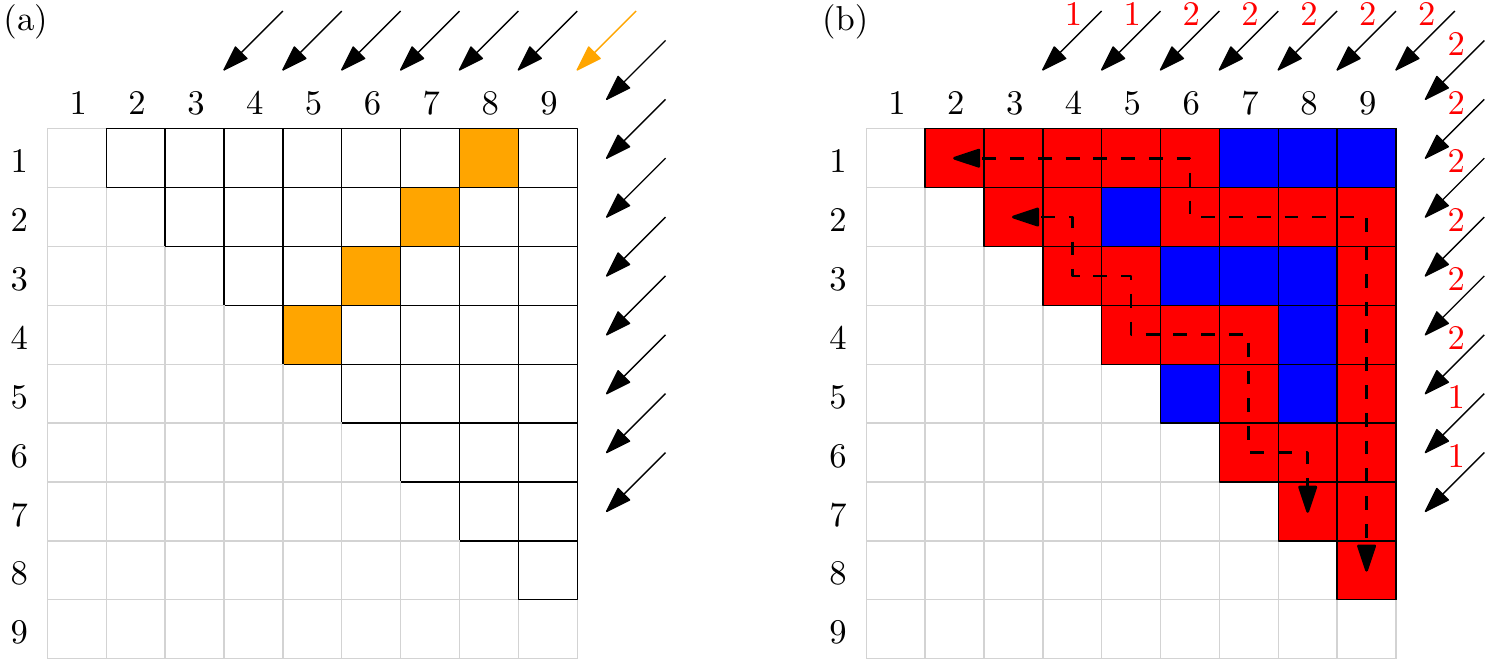}
\caption{(a) An example of an anti-diagonal in the matrix $A$ for $N=9$. (b) A construction of two disjoint routes in $A$ for $N=9$ and $n=3$, where red entries are 1-entries and blue entries are 0-entries. Each anti-diagonal achieves the maximum possible number of 1-entries, which leads to an ordered graph that does not contain a copy of $NM_3^<$ and has the maximum number of edges. The dashed paths denote the routes without their endpoints on the diagonal.}
\label{fig:anti-diagonals}
\end{figure}

It follows that there can be at most $(n-1)(2N-3)$ edges in $G^<$. To obtain a stronger estimate, we take into account the fact that, for each $k<2n-1$, the $k$th anti-diagonal contains only at most $\lfloor \frac{k-1}{2}\rfloor$ 1-entries. A similar estimate holds for the number of $1$-entries of a $k$th anti-diagonal with $k > 2N-2n+3$. Altogether, by summing the estimates for all anti-diagonals, we see that the number of edges of $G^<$ is at most
\begin{align*}
&\sum_{k=3}^{2n-2} \left \lfloor \frac{k-1}{2} \right \rfloor + \sum_{k=2n-1}^{2N-2n+3} (n-1) + \sum_{k=2N-2n+4}^{2N-1} \left \lfloor \frac{2N-k+1}{2} \right \rfloor \\
&= 2\sum_{k=1}^{n-2} k + (2N-4n+5)(n-1) + 2\sum_{k=1}^{n-2} k \\
&=(n-2)(n-1) + (2N-4n+5)(n-1) + (n-2)(n-1) \\
&=(n-1)(2N-2n+1),
\end{align*}
where we used the assumption $N\geq 2n$.

This upper bound is tight, which can be seen by considering the ordered graph $H^<$ on $N$ vertices with all edges of length at most $2n-2$.
Here, the \emph{length} of an edge $\{i,j\} \in \binom{N}{2}$ is defined as $|i-j|$.
Summing over the lengths $k$ of the edges of~$H^<$, we see that $H^<$ has exactly
$$\sum_{k=1}^{2n-2} (N-k) = (n-1)(2N-2n+1)$$
edges. The ordered graph $H^<$ then does not contain $NM^<_n$ as an ordered subgraph, as the longest edge in each copy of $NM^<_n$ in $H^<$ has to have length at least $2n-1$.
\end{proof}

A more general construction to achieve this maximum number of edges in an ordered graph without a copy of $NM^<_n$ is to lead $n-1$ pairwise disjoint ``routes'' in the matrix $A$ and set all entries of $A$ with positions in these routes to $1$ and all other entries of $A$ to $0$; see part~(b) of Figure~\ref{fig:anti-diagonals}.
For $k \in [n-1]$, the $k$th \emph{route} in the matrix $A$ is the set of positions $\{(i_\ell,j_\ell) \colon \ell \in [2N-4k+3]\}$ such that we have $(i_1,j_1) = (k,k)$, $(i_{2N-4k+3},j_{2N-4k+3}) = (N-k+1,N-k+1)$, and $(i_{\ell+1} = i_\ell+1 \; \& \; j_{\ell+1} = j_\ell)$ or $(i_{\ell+1} = i_\ell \; \& \; j_{\ell+1} = j_\ell+1)$ for every $\ell \in [2N-4k+2]$.
The third condition says that a route ``goes'' only to the right or down in $A$, while the first two conditions only specify the start and the end of each route.
Note that a route can contain entries of $A$ that lie on or below the main diagonal of $A$.
We say that an ordered graph $G^<$ is \emph{covered} by a set $R$ of routes if every position in the matrix representation of $G^<$ that corresponds to an edge of $G^<$ is contained in some route from $R$.

We note that a route in a matrix $A$ corresponds to a \emph{queue} in a \emph{linear queue layout} that is represented by $A$, as defined in~\cite{dujWoo10}.
Here, we stick to routes in matrix representations, as we find it more convenient to visualize red-blue coloring of complete ordered graphs with matrices.

We prove that there is no copy of $NM^<_n$ in the ordered graph that is covered by $n-1$ pairwise disjoint routes.
This result will be used several times later and it can be shown that if all such routes lie above the main diagonal of $A$, then we have an ordered graph without $NM^<_n$ containing the maximum number of edges.

\begin{lemma}
\label{lem-routes}
For every $n \in \mathbb{N}$, every ordered graph $G^<$ that is covered by $n-1$ pairwise disjoint routes does not contain $NM^<_n$ as an ordered subgraph.
\end{lemma}
\begin{proof}
Suppose for contradiction that $G$ contains a copy of $NM^<_n$. 
Every edge of this copy then belongs to one of the pairwise disjoint routes. No two distinct edges of a nested matching can belong to the same route, since a copy of $NM^<_2$ corresponds to two entries $(i,j)$ and $(k,l)$ of the matrix $A$ such that $i<k<l<j$, whereas any two entries $(i,j)$ and $(k,l)$ of a route satisfy $i\leq k$ and $j\leq l$ (or $k\leq i$ and $l\leq j$). However, by the pigeonhole principle, there is at least one such route with at least two edges of $NM^<_n$, a contradiction.
\end{proof}

By Lemma~\ref{lem-routes},
any ordered graph $H^<$ on $N$ vertices whose edges can be partitioned into $n-1$ pairwise disjoint routes does not contain $NM^<_n$.
Moreover, if the routes lie above the main diagonal of its matrix representation, then $H^<$ has $\sum_{k=1}^{n-1}(2N-4k+1) =(n-1)(2N-2n+1)$ edges, which is tight by Lemma~\ref{lemma:disjoint_paths}.

We are now ready to improve the upper bound from Proposition~\ref{nm}. Let us assume we have a red-blue coloring $\chi$ of the edges of $K^<_N$ with no red copy of $NM_n^<$ and no blue copy of $K^<_3$. Since there is no blue copy of $K^<_3$, there cannot be any vertex $v$ of $K^<_N$ contained in at least $2n$ blue edges, as otherwise there are only red edges between any two such neighbors of~$v$, which gives a red copy of $K^<_{2n}$ in $\chi$ and thus also a red copy of $NM^<_n$ in $\chi$.

If every vertex of $K^<_N$ is contained in less than $2n$ blue edges, then there is less than $2nN/2 = nN$ blue edges in $\chi$. Therefore there are more than $\binom{N}{2}-nN$ red edges. However, by Lemma~\ref{lemma:disjoint_paths}, there can be at most $(n-1)(2N-2n+1)$ red edges in $\chi$, as otherwise we have a red copy of $NM^<_n$ in $\chi$. Therefore we obtain
\[\frac{N(N-1)}{2}-nN \leq (n-1)(2N-2n+1),\]
which can be rewritten as
\[N^2+3N(1-2n)+(4n^2-6n+2) \leq 0.\]

By solving this quadratic inequality for $N\in\mathbb{N}$ we arrive at the bound \[N \leq \frac{1}{2}\left(\sqrt{20n^2-12n+1} + 6n-3\right).\]

For $n \geq \frac{1}{2}$, the right side of the above expression is at most $\left(3+\sqrt 5\right)n$, which concludes the proof of Theorem~\ref{prop:nm5.3}.

\subsection{Proof of Theorem~\ref{thm:nm_counterexamples}}

We show that, for every $n\geq6$, we have $r_<(NM^<_n,K^<_3) \geq 4n+1$.
Moreover, we prove $r_<(NM^<_4,K^<_3) = 16$ and $r_<(NM^<_5,K^<_3) = 20$. 
First, we prove the lower bound $r_<(NM^<_n,K^<_3) \geq 4n+1$ by showing that there exists a red-blue coloring $\chi$ of the edges of the complete ordered graph on $4n$ vertices without a blue triangle and a red copy of $NM^<_n$.

Let $n \geq 6$ be an integer.
The matrix representation $A$ of the coloring $\chi$ is illustrated in Figures~\ref{fig:obecna_konstrukce_24} and~\ref{fig:obecna_konstrukce_28}.
We now describe the construction of $\chi$ formally by listing all its blue edges.
We note that this coloring is \emph{symmetric}, that is, for all $i,j\in\mathbb{N}$ with $1\leq i<j\leq N$, the edges $\{i,j\}$ and $\{N-j+1, N-i+1\}$ have the same color.

\begin{figure}[ht]
\centering
\includegraphics[scale=0.65]{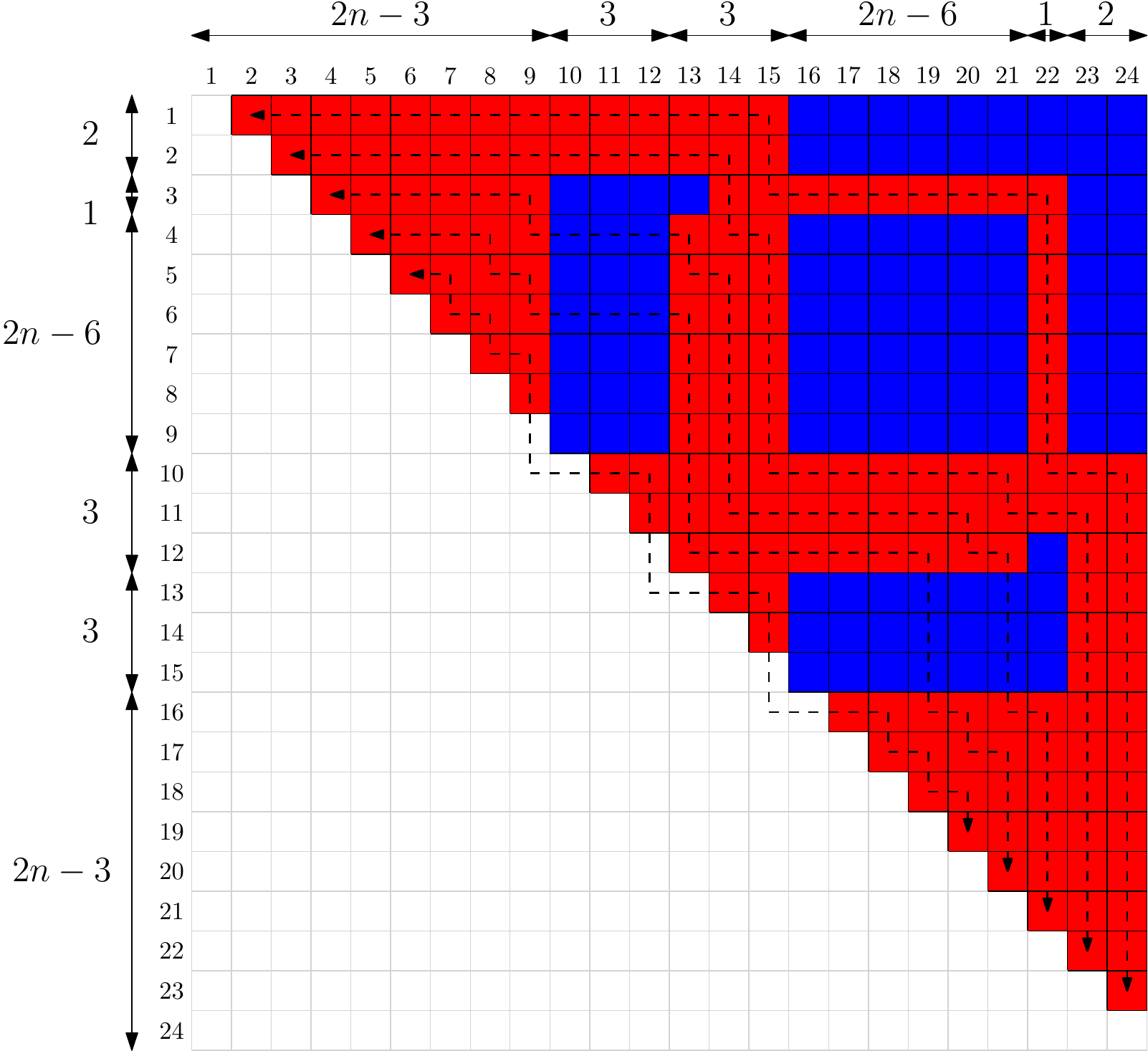}
\caption{The matrix representation of the coloring $\chi$ of the edges of $K^<_{4n}$ for $n=6$.}
\label{fig:obecna_konstrukce_24}
\end{figure}

\begin{figure}[ht]
\centering
\includegraphics[scale=0.65]{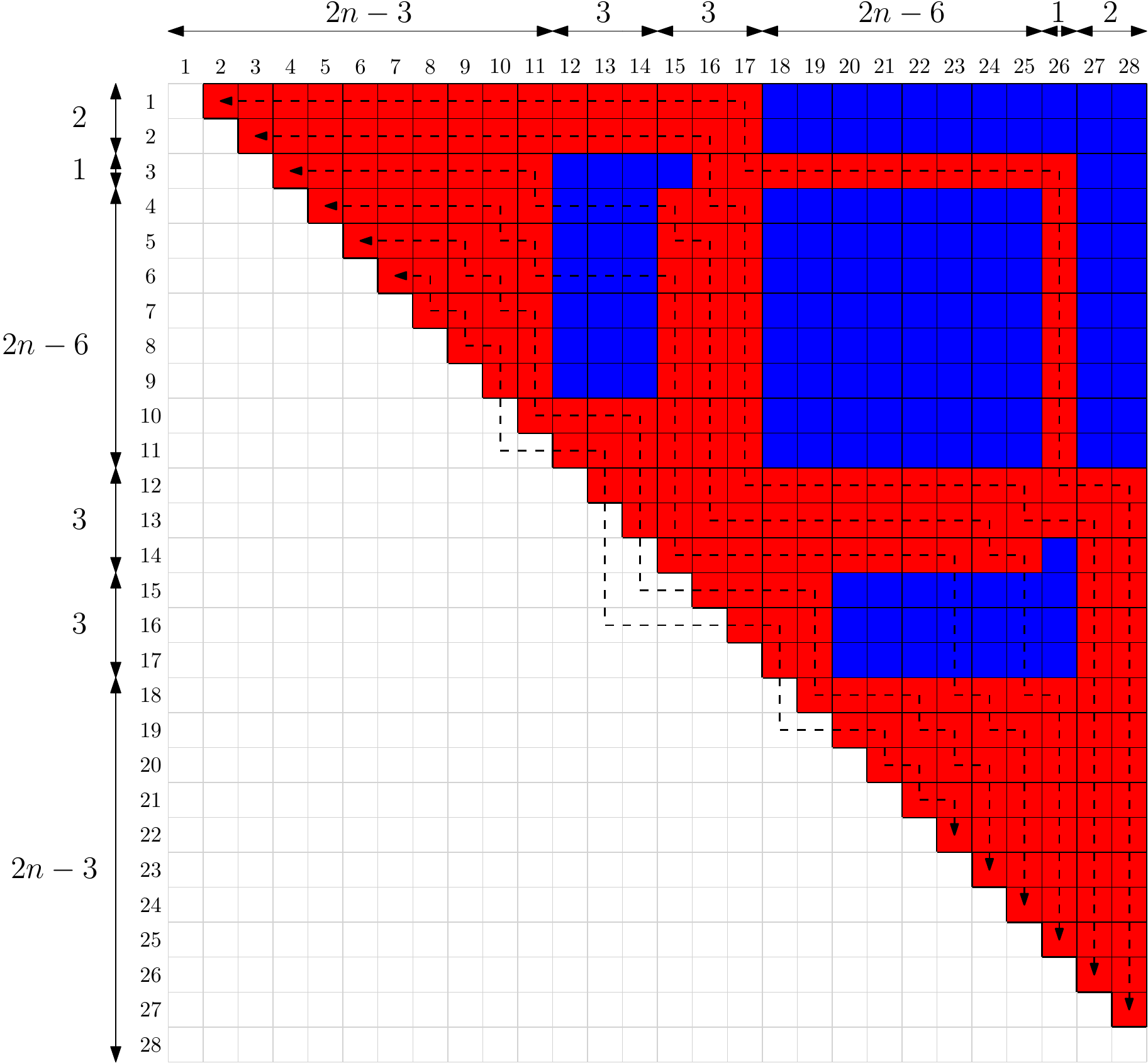}
\caption{The matrix representation of the coloring $\chi$ of the edges of $K^<_{4n}$ for $n=7$.}
\label{fig:obecna_konstrukce_28}
\end{figure}

The blue edges in $\chi$ are decomposed into the following sets: the set $S=\{\{i,j\} \colon 4 \leq i \leq 2n-3, 2n+4\leq j \leq 4n-3 \}$, which forms a $(2n-6) \times (2n-6)$ square in $A$, the set $L=\{\{i,j\} \colon i\in\{1,2\}, 2n+4\leq j \leq 4n \} \cup \{\{i,j\} \colon 1 \leq i \leq 2n-3, j\in\{4n,4n-1\}\}$, which corresponds to the L-shaped upper right corner of $A$ of width and height $2n-3$, and two sets $R_1=\{\{i,j\} \colon 3 \leq i \leq 9, 2n-2\leq j \leq 2n \}$ and $R_2=\{\{i,j\} \colon 2n+1 \leq i \leq 2n+3, 4n-8\leq j \leq 4n-2 \}$, which form two $3 \times 7$ rectangles in~$A$.
Finally, there are two single blue edges $e_1=\{3,2n+1\}$ and $e_2=\{2n,4n-2\}$.
All the remaining edges of $K^<_{4n}$ are red in $\chi$.

We show that the red edges of $\chi$ can be covered by $n-1$ pairwise disjoint routes.
Then it will follow from Lemma~\ref{lem-routes} that there is no red copy of the nested matching $NM^<_n$ in~$\chi$.
The set of routes covering the ordered graph formed by red edges in $\chi$ is constructed inductively with respect to $n$.
As the basis of the induction, we use the set of routes for $n=6$ that is illustrated in Figure~\ref{fig:obecna_konstrukce_24}.
For $n \geq 7$, we use essentially the same $n-2$ routes we had for $n-1$, we only elongate them. 
However, we additionally have to cover two new diagonals formed by entries on positions $(i,j)$ with $j-i \in \{1,2\}$ and $n-1 \leq i \leq 3n+1$; see Figure~\ref{fig:obecna_konstrukce_28}.
Covering these two new diagonals by an $(n-1)$st route is clearly possible and thus we can cover the whole ordered graph by $n-1$ pairwise disjoint routes.
Note that some entries of the two new diagonals might be covered by the first $n-2$ routes, but this makes covering their entries by the $(n-1)$st route only simpler.

To prove that $\chi$ does not contain a blue triangle for any $n\geq 6$, we consider the ordered graph formed by edges that are blue in $\chi$.
First, there is no blue triangle containing the edge $e_1=\{3,2n+1\}$, as in any such blue triangle there is another blue edge incident to vertex 3.
However, all other edges containing vertex 3 are of the form $\{3,i\}$ for $i \in \{2n-2,2n-1,2n\}\cup \{4n-1,4n\}$ and there is no blue edge of the form $\{2n+1,i\}$ for these $i$.
By symmetry, there is no blue triangle containing the edge $e_2=\{2n,4n-2\}$.

The edges from $S\cup L$ form a bipartite graph and thus there is no blue triangle with vertices in $S\cup L$ and any blue triangle in $\chi$ has to have an edge in $R_1 \cup R_2$.
Since both sets $R_1$ and $R_2$ induce a bipartite graph, any blue triangle in $\chi$ contains at most one edge in $R_1$ and at most one edge in $R_2$.

Consider a blue triangle $T$ with an edge from $R_1$.
By the definition of $R_1$, this edge contains a vertex $i \in \{2n-2,2n-1,2n\}$.
Since there is at most one edge of $T$ in $R_1$, there is an edge $\{i,j\}$ of $T$ that is not contained in $R_1$.
The vertex $j$ satisfies $j > i$, as all blue edges $\{i,k\}$ with $k \leq i$ lie in $R_1$.
However, the only blue edge of this form is for $i=2n$ and $j=4n-2$, which gives the edge $e_2$ and we already know that $e_2$ is not contained in a blue triangle.
Thus there is no blue triangle with an edge in $R_1$.
By symmetry, there is also no blue triangle with an edge from $R_2$ and, altogether, $\chi$ contains no blue triangle.

It is likely that our construction can be modified to obtain stronger lower bounds on $r_<(NM^<_n,K^<_3)$.
However, the coloring $\chi$ is easy to describe for any $n \geq 6$ and one can show that it does not contain the forbidden monochromatic ordered subgraphs without employing too complicated case analysis.
We also note that some of the blue edges might be colored red without introducing a red copy of $NM^<_n$ in the resulting coloring.

Now, we prove the rest of the statement of Theorem~\ref{thm:nm_counterexamples}, that is, we show $r_<(NM^<_4,K^<_3) = 16$ and $r_<(NM^<_5,K^<_3) = 20$.
The lower bounds $r_<(NM^<_4,K^<_3) \geq 16$ and $r_<(NM^<_5,K^<_3) \geq 20$ are obtained using red-blue colorings $\chi_1$ and $\chi_2$ whose matrix representations can be seen in parts~(a) and~(b) of Figure~\ref{fig:counterexample1}, respectively.
We prove that the colorings $\chi_1$ and $\chi_2$ do not contain a red copy of $NM^<_4$ and $NM^<_5$, respectively, and that there is also no blue triangle.

\begin{figure}[ht]
\centering
\includegraphics[scale=0.65]{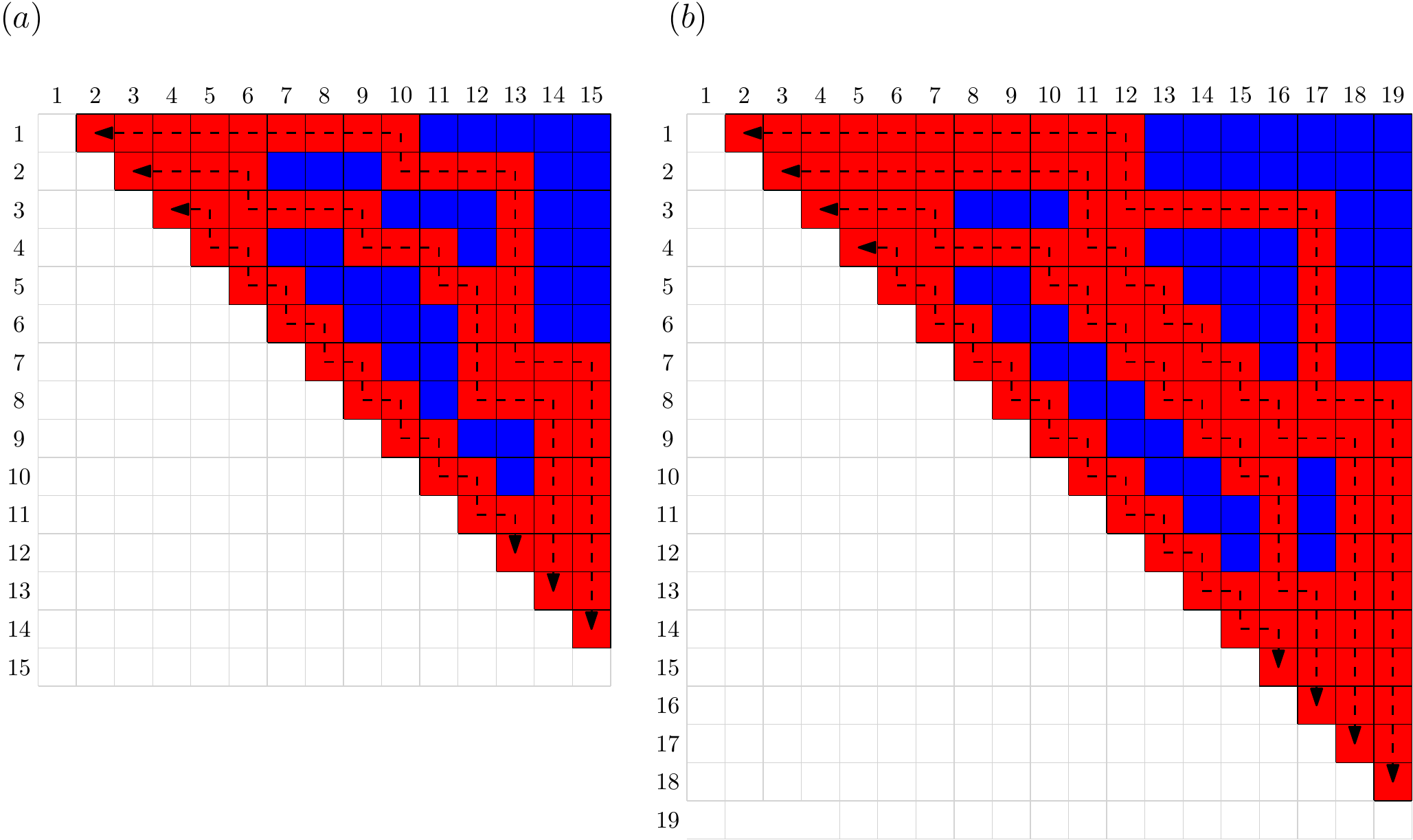}
\caption{(a) The coloring $\chi_1$ of the edges of $K^<_{15}$ vertices without a red $NM^<_4$ and a blue $K^<_3$.
	(b) The symmetric coloring $\chi_2$ of the edges of $K^<_{19}$ without a red $NM^<_5$ and a blue $K^<_3$.}
	\label{fig:counterexample1}
\end{figure}

The red edges of $\chi_1$ and $\chi_2$ can be partitioned into three and four disjoint routes, respectively, in the matrix representation (see Subsection~\ref{subsec-prop:nm5.3} for the definition).
It follows from Lemma~\ref{lem-routes} that $\chi_1$ and $\chi_2$ do not contain a red copy $NM^<_4$ and $NM^<_5$, respectively, as an ordered subgraph.
Note that all the routes lie above the main diagonal and thus both examples contain the maximum possible number of red edges by Lemma~\ref{lemma:disjoint_paths}.

Proving that $\chi_1$ does not contain a blue triangle is a bit tedious and we omit it here, as we verified it using a computer.
We can, however, use the fact that $\chi_2$ is symmetric to give a short explicit proof of the fact that $\chi_2$ contains no blue copy of $K^<_3$.

Let $S=\{1,\dots,10\}$ and suppose for contradiction there exists a blue copy of $K^<_3$ in $\chi_2$.
Because of symmetry reasons, we may assume there is a copy of $K^<_3$ formed by vertices $v_1<v_2<v_3$ such that $v_1,v_2 \in S$.
Vertices $1,2,4$ do not have any blue neighbors in $S$. The only blue neighbors of vertex 3 in $S$ are $8,9,10$ and the only blue neighbors of vertex 3 outside of $S$ are $18$ and $19$.
However, there is no blue edge between vertices from $\{8,9,10,18,19\}$, hence vertex 3 also cannot be a part of any blue triangle.
Thus $v_1, v_2 \in \{5,6,\dots,10\}$.

Blue neighbors of vertex 8 are $3,5,11,12$ and these vertices form a red clique.
Similarly, blue neighbors of vertex 9 are $3,5,6,12,13$ and also form a red clique.
Blue neighbors of 10 are two sets of vertices $3,6,7$ and their reflections $13,14,17$.
These six vertices also form a red clique.
The only option left is $v_1,v_2\in\{5,6,7\}$, but there are no blue edges between these vertices, a contradiction.
The coloring $\chi_2$ thus contains no blue triangle.

The proof of the upper bounds $r_<(NM^<_4,K^<_3) \leq 16$ and $r_<(NM^<_5,K^<_3) \leq 20$ is computer assisted.
We performed an exhaustive computer search using our SAT-solver-based program~\cite{utility} and verified that every red-blue coloring of the edges of the ordered complete graph on $16$ vertices contains either a red copy of $NM_4^<$ or a blue triangle.
Similarly, every red-blue coloring of the edges of $K^<_{20}$ contains either a red copy of $NM_5^<$ or a blue triangle.

\subsection{Proofs of Theorem~\ref{prop:nm2} and Corollary~\ref{cor-chromatic}}
\label{subsec-nm2}

First, we prove Theorem~\ref{prop:nm2} by deriving the exact formula $r_<(NM_2^<,K^<_n) = 3n-2$ for the ordered Ramsey number of the nested matching $NM^<_2$ versus $K^<_n$.
The lower bound $r_<(NM_2^<,K^<_n) \geq 3n-2$ follows from a simple red-blue coloring of the edges of $K^<_{3n-3}$, where we partition the vertex set into $n-1$ consecutive cliques, each of size 3, and we color all edges between vertices from the same clique red and all other edges blue.
Then it is easy to see that there is no red copy of $NM^<_2$ and no blue copy of $K^<_n$.

To show the upper bound $r_<(NM_2^<,K^<_n) \leq 3n-2$, we first state an auxiliary result about ordered graphs that do not contain $NM^<_m$ as an ordered subgraph for some $m \in \mathbb{N}$.
This result was proved by Dujmović and Wood~\cite{dujWoo10}.

\begin{lemma}[\cite{dujWoo10}]
\label{lem-1queue}
Every 1-queue graph is 3-colorable.
\end{lemma}

Now, let $\chi$ be a red-blue coloring of the edges of $K^<_{3n-2}$.
Assume that $\chi$ does not contain a red copy of $NM^<_2$.
The ordered graph $R^<$ formed by edges that are red in $\chi$ is a 1-queue graph, as $R^<$ does not contain a copy of $NM^<_2$.
By Lemma~\ref{lem-1queue}, the ordered graph $R^<$ is 3-colorable and therefore we can partition its vertex set into 3 disjoint sets $C_1,C_2,C_3$ such that no two vertices from the same set are connected by an edge in $R^<$.

By the pigeonhole principle, there is a set $C_i$ that contains at least $\left \lceil{\frac{3n-2}{3}}\right \rceil = n$ vertices.
Since there is no edge of $R^<$ between any two vertices from $C_i$, we see that $C_i$ induces a blue copy of $K^<_n$ in $\chi$, which gives the desired upper bound $r_<(NM_2^<,K^<_n) \leq 3n-2$ and finishes the proof of Theorem~\ref{prop:nm2}.

In the rest of the subsection, we prove Corollary~\ref{cor-chromatic} by showing that the maximum chromatic number $\chi_k$ of $k$-queue graphs is at least $2k+2$ for every $k \geq 3$.

\begin{proof}[Proof of Corollary~\ref{cor-chromatic}]
For $k \in \mathbb{N}$, let $N$ be a positive integer such that $r_<(NM^<_{k+1},K^<_3) > N$.
We will show that $\chi_k \geq \lceil N/2\rceil$.
Since $r_<(NM^<_{k+1},K^<_3) > N$, there is a red-blue coloring of the edges of $K^<_N$ without a red copy of $NM^<_{k+1}$ and a blue copy of $K^<_3$.
Let $R^<$ be the ordered subgraph of $K^<_N$ formed by red edges.
Suppose for contradiction that the chromatic number $\chi(R^<)$ of $R^<$ is less than $\lceil N/2\rceil$.
Then, by the pigeonhole principle, there is an independent set in $R^<$ of size $s \geq 3$.
However, since there is no blue copy of $K^<_3$, we have $s < 3$, a contradiction.

By Theorem~\ref{thm:nm_counterexamples}, we have $r_<(NM^<_{k+1}K^<_3) > 4(k+1)-1$ for every $k \geq 3$.
Applying this estimate to the previous observation, we obtain $\chi_k \geq \left\lceil \frac{4(k+1)-1}{2} \right\rceil = 2k+2$.
\end{proof}

\subsection{Proof of Theorem~\ref{thm:linearni_matchingy}}

Here, we prove Theorem~\ref{thm:linearni_matchingy} by showing $r_<(NM^<_m,K^<_{n+1}) = \Theta(mn)$ for every $m,n\in\mathbb{N}$.
We recall the famous \emph{Turán's theorem}~\cite{turan}, which states that, for any $n,N\in\mathbb{N}$, every graph $G$ on $N$ vertices that does not contain $K_{n+1}$ as a subgraph has at most
$\left(1-\frac{1}{n}\right)\frac{N^2}{2}$ edges.

Let $\chi$ be a red-blue coloring of the edges of $K^<_N$ , which does not contain a red copy of~$NM^<_m$ nor a blue copy of $K^<_{n+1}$.
We proceed along the lines of our proof for Proposition~\ref{prop:nm5.3}.
By Lemma~\ref{lemma:disjoint_paths}, there can be at most
$(m-1)(2N-2m+1)$ red edges in $\chi$. 
Since the subgraph formed by blue edges does not contain a copy of $K_{n+1}$, Turán's theorem implies that there can be at most 
$\left(1-\frac{1}{n}\right)\frac{N^2}{2}$
blue edges in $\chi$. Thus the following inequality 
\[(m-1)(2N-2m+1) + \left(1-\frac{1}{n}\right)\frac{N^2}{2} \geq \frac{N(N-1)}{2}\] is satisfied and it can be rewritten as
\[N^2 - Nn(4m-3) + n(4m^2-6m+2) \leq 0.\]

By solving this quadratic inequality, we get 
\[N \leq \frac{1}{2}\left(4mn-3n + \sqrt{(4m-3)^2n^2-16m^2n+24mn-8n}\right)\]
and thus $N = O(mn)$. 

The lower bound $N = \Omega(mn)$ can be obtained from a coloring on $(2m-1)(n-1)$ vertices formed by $n-1$ red cliques formed by consecutive vertices, each of size $2m-1$, such that any two vertices from different cliques form a blue edge. This coloring clearly contains no red copy of $NM^<_m$ and no blue copy of $K^<_n$ as ordered subgraphs.

\section{Ramsey goodness for ordered graphs}

This section is devoted to proofs of results of about Ramsey goodness of ordered graphs that are stated in Subsection~\ref{subsec:goodness}.
That is, we prove Theorem~\ref{thm:expansion} and Proposition~\ref{prop:forbidden_caterpillar}.

\subsection{Proof of Theorem~\ref{thm:expansion}}

Here, we present the proof of Theorem~\ref{thm:expansion} by showing that a join of a good ordered graph with a one-sided star is a good graph.
Formally, we prove that, for all $n, r, l\in\mathbb{N}$, if a connected ordered graph $G^<$ is $n$-good, then the ordered graphs $G^< + S^<_{1,r}$, $G^< + S^<_{l,1}$, $S^<_{l,1}+G^<$, and $S^<_{1,r}+G^<$ are also $n$-good.
We start with the following auxiliary result.

\begin{lemma}
\label{lemma:copies}
Let $G^<$ be a connected ordered graph such and set $N = r_<(G^<,K^<_n)$.
Then every red-blue coloring of the edges of $K^<_M$, where $M \geq N$, contains either a blue copy of $K^<_n$ or there are $M-N+1$ vertices of $K^<_M$ such that each one of them is the rightmost vertex of some red copy of $G^<$.
\end{lemma}
\begin{proof}
Let $\chi_1$ be a red-blue coloring of the edges of $K^<_M$, where $M \geq N$.
By the definition of~$N$ and since $M \geq N$, the coloring $\chi_1$ contains either a red copy of $G^<$ or a blue copy of~$K^<_n$ as an ordered subgraph.
In the latter case we are finished.
In the former case, we can find a red copy of $G^<$ in $\chi_1$ and delete rightmost vertex $v_1$ from $K^<_M$, obtaining a red-blue coloring $\chi_2$ of the edges of $K^<_{M-1}$ vertices.
We can iterate this approach as long as the number of vertices is at least $N$.
For every $i=1,\dots,M-N+1$, we either find a blue copy of $K^<_n$ or a vertex $v_i$ that is the rightmost vertex of some red copy $G_i^<$ of $G^<$ in coloring $\chi_i$ that is obtained $\chi_{i-1}$ by removing $v_{i-1}$.
At the end, the vertices $v_1,\dots,v_{M-N+1}$ satisfy the statement of the lemma. 
\end{proof}

Assume that we have a connected $n$-good ordered graph $G^<$ with $k+1$ vertices.
It is sufficient to prove that the two ordered graphs $G^< + S^<_{1,m-k}$, $G^< + S^<_{m-k,1}$, where $m$ is an arbitrary integer larger than $k$, are $n$-good, as the remaining two cases from the statement of Theorem~\ref{thm:expansion} follow by symmetry. 
Let us denote $F^<_1 = G^< + S^<_{1,m-k}$ and $F^<_2 = G^< + S^<_{m-k,1}$ and note that $|F^<_1|=|F^<_2|=m$. 

Let $\chi$ be a red-blue coloring of the edges of $K^<_N$ for $N=(m-1)(n-1)+1$.
We will prove that $\chi$ contains either red copies of both $F^<_1$ and $F^<_2$ or a blue copy of $K^<_n$ as ordered subgraphs.
Since $G^<$ is $n$-good, we have $r_<(G^<,K^<_n)=k(n-1)+1$ and, by Lemma~\ref{lemma:copies}, there is a set $W$ of 
\[((m-1)(n-1)+1) - (k(n-1) +1) +1=(n-1)(m-k-1)+1\]
vertices such that each one of them is the rightmost vertex of at least one red copy of $G^<$ in $\chi$.

Assume that the coloring $\chi$ contains no red copy of $F^<_1$.
We show that $\chi$ then contains a blue copy of $K^<_n$.
We set $W_0=W$ and, for every $i=1,\dots,n-1$, we define $W_i$ as the set of all right blue neighbors in $W_{i-1}$ of the leftmost vertex $w_i$ from $W_{i-1}$.
Since there is no red copy of $F^<_1$, each $w_i$ has at most $(m-k-2)$ red neighbors to the right.
Therefore, in each transition from $W_{i-1}$ to $W_i$, at most $(m-k-1)$ vertices are removed and we have $|W_i| \geq |W_{i-1}| - (m-k-1)$ for every $i=1,\dots,n-1$.
Since $|W_0|=|W|=(n-1)(m-k-1)+1$, the set $W_{n-1}$ contains at least
\[(n-1)(m-k-1)+1 - (n-1)(m-k-1) = 1\]
vertices.
In particular, there is at least one vertex in $W_{n-1}$ and we set $w_n$ to be an arbitrary one of them.
By the choice of the sets $W_0,\dots,W_{n-1}$, there is no red edge in $\chi$ between any two vertices from $\{w_1,\dots, w_n\}$ and therefore they induce a blue copy of $K^<_n$ in $\chi$.

To show that there is a red copy of $F^<_2$ in $\chi$, we proceed analogously.
Assume that $\chi$ contains no such red copy.
We then set $W'_0=W$ and, for every $i=1,\dots,n-1$, define $W'_i$ as the set of all \emph{left} blue neighbors in $W'_{i-1}$ of the \emph{rightmost} vertex $w'_i$ from $W'_{i-1}$. We know that each $w'_i$ has at most $(m-k-2)$ red neighbors to the left, as there is no red copy of $F^<_2$ in $\chi$.
The same calculation as before then gives $|W'_{n-1}|\geq 1$ and we let $w'_n$ be an arbitrary vertex from $W'_{n-1}$.
The vertices $w'_1, \dots, w'_n$ then again induce a blue copy of $K^<_n$  in $\chi$.

\subsection{Proof of Proposition~\ref{prop:forbidden_caterpillar}}

In this subsection, we prove that a connected ordered graph $G^<$ is a monotone caterpillar graph if and only if $G^<$ does not contain any of the four ordered graphs from Figure~\ref{fig:forbidden_caterpillar} as an ordered subgraph.
The implication from left to right is trivial, because, by definition, a monotone caterpillar graph cannot contain any of the ordered graphs from Figure~\ref{fig:forbidden_caterpillar} as an ordered subgraph.

To prove the other implication, we show that any ordered graph that does not contain any of the four ordered graphs from Figure~\ref{fig:forbidden_caterpillar} is a monotone caterpillar graph.
We proceed by induction on the number of vertices.

The statement is trivial for connected ordered graphs with at most 3 vertices and it can be easily checked by hand for connected ordered graphs with exactly 4 vertices.
Now, for $n \geq 4$, assume the statement is true for all connected ordered graphs with at most $n$ vertices and let $G^<$ be a connected ordered graph with vertices $1,\dots,n+1$ such that $G^<$ does not contain any ordered graph from Figure~\ref{fig:forbidden_caterpillar} as an ordered subgraph.

Since $G^<$ is connected, the last vertex $n+1$ has at least one neighbor and we let $v$ be the leftmost neighbor of $n+1$.
The ordered subgraph $H^<$ of $G^<$ induced by vertices $1,\dots,v$ is a monotone caterpillar graph by the induction hypothesis.
If $v=n$, then $G^<=H^<+S_{1,2}$ and thus $G^<$ is also a monotone caterpillar graph.

If $v<n$, then we let $I=\{v+1,\dots,n\}$ and we note that $I$ is non-empty.
There are no edges between any two vertices from $I$, as otherwise $G^<$ contains $A=NM^<_2$ as an ordered subgraph.
There are also no edges going from $I$ to the left of $v$ as this would imply $B$ as an ordered subgraph. 
Since the ordered graph $G^<$ is connected, all vertices from $I$ thus are adjacent either to $v$ or $n+1$, but not to both, as otherwise we would get a copy of $D=K_3^<$ as an ordered subgraph of $G^<$.
Suppose some $v_1\in I$ is adjacent to $v$ and some other $v_2\in I$ is adjacent to $n+1$.
If $v_1 > v_2$, then the vertices $v,v_1,v_2,n+1$ induce a copy of $B$ in $G^<$.
On the other hand, if $v_1<v_2$, then  $v,v_1,v_2,n+1$ induce a copy of $C$ in $G^<$.
Since both $B$ and $C$ are forbidden in $G^<$, all vertices from $I$ are adjacent either to $n+1$ or all of them are adjacent to $v$.
In both cases $G^<$ is a monotone caterpillar graph, as either $G^<=H^<+S_{n-v+2,1}$ in the former case, or $G^<=H^<+S_{1,n-v+2}$ in the latter case.

\bibliography{bibliography}	
\bibliographystyle{plain}

\end{document}